\documentclass{article}
\usepackage{epsfig,amsmath,amsthm,amssymb,graphics,verbatim,amsfonts,subfigure,psfrag}

\usepackage{makeidx}

\theoremstyle{plain}
\newtheorem{thm}{Theorem}[section]

\newtheorem{prop}[thm]{Proposition}

\theoremstyle{definition}

\def\R{\mathbb{R}}
\def\C{\mathbb{C}}

\def\I{\infty}

\newcommand{\be}{\begin{equation}}
\newcommand{\ee}{\end{equation}}
\newcommand{\benn}{\begin{equation*}}
\newcommand{\eenn}{\end{equation*}}
\newcommand{\bea}{\begin{eqnarray}}
\newcommand{\eea}{\end{eqnarray}}
\newcommand{\beann}{\begin{eqnarray*}}
\newcommand{\eeann}{\end{eqnarray*}}

\begin{document}
 
\date{}
\title{Scaling of Saddle-Node Bifurcations:\\ Degeneracies and Rapid Quantitative Changes}
\author{Christian Kuehn
\thanks{Center for Applied Mathematics, Cornell University, Ithaca, NY 14853}}

\maketitle

\begin{abstract}
The scaling of the time delay near a ``bottleneck'' of a generic saddle-node bifurcation is well-known to be given by an inverse square-root law. We extend the analysis to several non-generic cases for smooth vector fields. We proceed to investigate $C^0$ vector fields. Our main result is a new phenomenon in two-parameter families having a saddle-node bifurcation upon changing the first parameter. We find distinct scalings for different values of the second parameter ranging from power laws with exponents in $(0,1)$ to scalings given by $O(1)$. We illustrate this rapid quantitative change of the scaling law by a an overdamped pendulum with varying length. 
\end{abstract}

\section{Introduction}

Saddle-node bifurcations have been extensively studied in dynamical systems. The normal form in the context of ordinary differential equations is 
\be
\label{eq:sn}
\dot{x}=r+x^2 \qquad \text{ for $x \in\R$ and $r\in\R$} 
\ee
where $r$ is the bifurcation parameter. Solving for the fixed points  we set $r+x^2=0$, which has two solutions $x_\pm=\pm\sqrt{-r}$ giving no fixed points for $r>0$ and a non-hyperbolic fixed point $x_{\pm}=0$ at $r=0$. For $r<0$ we obtain an attracting fixed point $x_-$ and a repelling fixed point at $x_+$, hence a saddle-node bifurcation occurs at $r=0$. If we set $x(0)=x_0$ we can solve (\ref{eq:sn}) by separation of variables and obtain
\be
\label{eq:sov0}
\int_{x_0}^{x(t)}\frac{1}{r+s^2}ds= \int_0^t dt
\ee 
Using trigonometric substitution we get
\be
\label{eq:sov1}
\tan^{-1}\left(\frac{x(t)}{\sqrt{r}} \right) - \tan^{-1}\left(\frac{x_0}{\sqrt{r}} \right) = t\sqrt{r}
\ee    
Therefore the time a trajectory spends in the interval $I=[-1,1]$ is
\be
\label{eq:sov2}
\frac{2\tan^{-1} \left(  \frac{1}{\sqrt{r}} \right)}{\sqrt{r}} = t
\ee
Taking $r\rightarrow 0^+$ gives
\be
\label{eq:sov3}
t  \sim  \frac{1}{\sqrt{r}}
\ee
Note that any other interval $I=[-\epsilon,\epsilon]$ for some $\epsilon>0$ would yield the same result. In the literature this scaling law is referred to as \textit{intermittency}, \textit{``bottleneck''} or \textit{saddle-node ghost}. One of the earliest references is \cite{Pomeau} where saddle-node bifurcations of maps are investigated. Textbook references include \cite{GH} and \cite{Strogatz}. Applications of the scaling law to physical systems can, for example, be found in \cite{BulsaraSQUID,El-NasharSYNC,StrogatzWestervelt,Trickey,Sardanyes1,Sardanyes2}. We remark that the square-root scaling is found for any generic saddle-node bifurcation in sufficiently smooth vector fields independent of the dimension of the phase space. To justify this statement, recall the following theorem (see \cite{GH}):

\begin{thm}
\label{thm:gsn}
Consider $\dot{y}=g(\mu,y)$ with $\mu\in \R$ and $y\in \R^n$. Assume that for $\mu=\mu_0$ there exists an equilibrium $y_0$ such that:
\begin{enumerate}
	\item $D_yg(\mu_0,y_0)$ has a simple eigenvalue $0$ with right eigenvector $v$ and left eigenvector $w$.
	\item $D_yg(\mu_0,y_0)$ has $k$ eigenvalues of negative real part and $n-k-1$ eigenvalues with positive real part.
 	\item $w((\partial g/\partial \mu)(y_0,\mu_0))\neq 0$
	\item $w(D_y^2 g(\mu_0,y_0)(v,v))\neq 0$
\end{enumerate}
then there is a smooth curve of equilibria in $\R\times \R^n$ passing through $(\mu_0,y_0)$ tangent to the hyperplane $\{\mu_0\} \times \R^n$ with no equilibrium on one side of the hyperplane for each value $\mu$ and $2$ equilibria on the other side of the hyperplane for each $\mu$. The two equilibria are hyperbolic and have stable manifolds of dimensions $k$ and $k+1$ respectively.  
\end{thm}    

Furthermore the set of equations satisfying Theorem \ref{thm:gsn} is known to be open and dense in $C^\I$ one-parameter families with an equilibrium having a simple zero eigenvalue. Using a center manifold reduction \cite{Kusnetzov,GH} we can restrict to one-dimencsional dynamics and consider the equation 
\be
\dot{x}=f(r,x) \qquad \text{for $x \in\R$ and $r\in\R$}\nonumber
\ee
From now on we shall consider one-dimensional flows only. Necessary conditions for a saddle-node at $(r,x)=(0,0)$ are $f(0,0)=0=f_x(0,0)$. The genericity and transversality conditions in this case are:
\bea
\label{eq:gn1}
f_{xx}(0,0)\neq 0\\
\label{eq:gn2}
f_r(0,0)\neq 0
\eea
If the necessary conditions and \eqref{eq:gn1}-\eqref{eq:gn2} are satisfied we call a saddle-node non-degenerate. We want to show explicitly that the scaling law for a non-degenerate saddle-node is given by the normal form \eqref{eq:sn}. First recall one special case of the Malgrange Preparation Theorem \cite{NirenbergMPT,LuSing}:

\begin{thm}
\label{thm:mpt}
Suppose $f(r,x)$ is a real-analytic function on $\R^2$ and
\benn
f(0,0)=0=\frac{\partial f}{\partial r}(0,0)=\ldots=\frac{\partial^{m-1} f}{\partial r^{m-1}}(0,0) \quad \text{and} \quad \frac{\partial f^m}{\partial r^m}(0,0)\neq 0.
\eenn 
Then there exists a smooth function $c(r,x)$ which is nonzero near the origin such that
\benn
f(r,x)=c(r,x)(r^m+a_{m-1}(x)r^{m-1}+\ldots+a_0(x))
\eenn
\end{thm} 

Applying the Malgrange Preparation Theorem to a non-degenerate saddle-node we find that it is locally given by
\benn
\dot{x}=c(r,x)(r+a(x))
\eenn
where $c(r,x)\neq 0$ for some neighbourhood $V=[-\epsilon,\epsilon]^2$. Also $a(x)$ satisfies $a(0)=0=a'(0)$ and $a''(0)\neq 0$. The same technique as used in \eqref{eq:sov0}-\eqref{eq:sov3} implies that the scaling law can be given calculating
\benn
\int_{-\epsilon}^{\epsilon}\frac{1}{c(r,x)(r+c_2x^2+O(x^3))}dx
\eenn 
Since $c(r,x)$ is bounded and nonzero on $[-\epsilon,\epsilon]$ we have that
\benn
\int_{-\epsilon}^{\epsilon}\frac{1}{c(r,x)(r+c_2x^2+O(x^3))}dx\sim \int_{-\epsilon}^{\epsilon}\frac{1}{r+c_2x^2+O(x^3)}dx\sim\int_{-\epsilon}^{\epsilon}\frac{1}{r+x^2}dx
\eenn   
as $r\rightarrow 0^+$. In the next section we briefly address the question, what happens to the scaling law when the saddle-node is degenerate.

\section{Degenerate Saddle-Nodes}
We treat each of the cases for degenerate saddle-nodes in turn. We shall assume from now on that the vector fields under consideration are analytic.

\subsection{Case 1 - $f_{xx}(0,0)= 0$}

Again using the Malgrange Preparation Theorem we can reduce to the case:
\be
\dot{x}=r+F(x) \qquad \text{for $x \in\R$ and $r\in\R$} \nonumber
\ee
with $F\in C^\omega(\R)$ and $F_{xx}(0,0)=0$. Notice that we still require that $F(0)=0=F_x(0)$ as necessary conditions for a saddle-node at $(r,x)=(0,0)$. Also we assume without loss of generality that $F(x)>0$ for $x\neq 0$ near $0$. Considering the Taylor expansion of $F$ at $0$ we get $F(x)=c_{2k}x^{2k}+O(x^{2k+1})$ for some $k>1$ as $F_{xx}(0)=0$; note that we can disregard the odd terms since if the leading term in the Taylor expansion is odd then we do not have a topological saddle-node. Since $F$ is analytic it has a a lowest order non-zero coefficient $c_{2k}\neq 0$ in its Taylor expansion. This excludes the ``completely flat'' case when all Taylor coefficients vanish. To find the scaling law for the time a trajectory spends in $[-1,1]$ we have to evaluate the integral:
\be
\int_{-1}^{1} \frac{1}{r+c_{2k}x^{2k} +O(x^{2k+1})} dx \sim \int_{-1}^{1} \frac{1}{r+x^{2k}} dx \qquad \text{for $r\rightarrow 0^+$} \nonumber
\ee
Since the asymptotic behavior of the time spend in $[-1,1]$ is independent of the interval of non-zero length centered at $0$ we can find the scaling law by solving the integration problem:
\be
\int_{-\I}^{\I} \frac{1}{r+x^{2k}}dx \qquad \text{for $k\in\{2,3,4,\ldots\}$} \nonumber
\ee
To evaluate the last integral we either use substitution twice or use contour integration over the contour in $\C$  given by the semi-circle with radius $a$ and the interval $[-a,a]$; see e.g. \cite{Fontich} where this calculation is carried out in detail. In any case, we get
\be
\int_{-\I}^{\I} \frac{1}{r+x^{2k}}dx \sim r^{-\frac{2k-1}{2k}} \qquad \text{as $r\rightarrow 0^+$} \nonumber
\ee
The discussion can be summarized in the following result:

\begin{prop}
\label{prop:dgprop2}
Consider the ODE
\be
\label{eq:dg2}
\dot{x}=r+F(x) \qquad \text{with $F(0)=0=F_x(0,0)=F_{xx}(0)$ and $F(x)>0$ for $x\neq0$}
\ee
Assume that $F$ is real analytic. Define $m$ to be the smallest exponent with nonzero coefficient in the Taylor expansion for $F$ at $0$, then (\ref{eq:dg2}) has a saddle-node bifurcation at $(r,x)=(0,0)$ with scaling law given by
\be
t\sim r^{-\frac{m-1}{m}} \nonumber
\ee
\end{prop}

\subsection{Case 2 - $f_r(0,0)= 0$}
The transversality condition for $r$ at $r=0$ fails and we can again apply the Malgrange Preparation Theorem to simplify the situation. Note that here the variables $(r,x)$ must be interchanged in the statement of Theorem \ref{thm:mpt} to get that
\be
\dot{x}=a_0(r)+a_1(r)x+x^2 \qquad \text{for $x \in\R$, $r\in\R$ and $R\in C^\omega(\R)$} \nonumber
\ee 
By a change of variable $x\mapsto x-a_1(r)/2$ we can further reduce this to 
\be
\label{eq:ntsn}
\dot{x}=R(r)+x^2 \qquad \text{for $x \in\R$, $r\in\R$ and $R\in C^\omega(\R)$}
\ee 
without restrictions on the first derivative of $R$ and $R(0)=0$. Therefore $R(r)$ has a Taylor expansion given by:
\be
\label{eq:TeR}
R(r)=d_kr^k+d_{k+1}r^{k+1}+\ldots
\ee
where $d_k\neq 0$ and $k>1$. In particular we can now show:

\begin{prop}
\label{prop:rsmooth}
Consider the differential equation \eqref{eq:ntsn},\eqref{eq:TeR} of a non-transversal saddle-node  Then the scaling law is given by 
\benn
t\sim \frac{1}{r^{k/2}}
\eenn 
\end{prop}

\begin{proof}
Note that $r$ is treated as a constant in the integral calculation \eqref{eq:sov0}-\eqref{eq:sov3}. Hence the scaling law is $t\sim (R(r))^{-1/2}$ for $r\rightarrow 0^+$. The lowest order term in the Taylor expansion for $R$ determines the asymptotics as $r\rightarrow 0^+$ and the result follows. 
\end{proof}

Remark: Note that we have included topologically degenerate cases such as $\dot{x}=r^2+x^2$ in Proposition \ref{prop:rsmooth}. Disregarding these cases excludes all scaling laws with even powers $k=2m$. 

\subsection{Further Remarks}

Although degenerate saddle-nodes are not dense in the space smooth vector fields they still might be observed in practical applications, e.g. due to symmetry inherent in the system or given nonlinear parameter dependencies. Propositions \ref{prop:dgprop2} and \ref{prop:rsmooth} show that we should expect various power laws as scalings near a (degenerate) saddle-node bifurcation for a smooth vector field. The natural question arises what happens if we drop the smoothness requirements of our vector field. In particular we consider the case when the vector field is only continuous at $(r,x)=(0,0)$. 

\section{Non-Smooth Saddle-Nodes}

We define a saddle-node bifurcation for $C^0$ vector fields by requiring orbital topological equivalence to the smooth saddle-node case. We remark that the analytical description presented in Theorem \ref{thm:gsn} no longer applies and refer to \cite{PSDS} and \cite{LeineNijmeijer} (and references therein) for analytic methods in the theory of non-smooth bifurcations. Suppose we drop the smoothness requirement on the parameter dependence and allow $R(r)\in C^0(\R)$, then we can show:

\begin{prop}
\label{prop:dgprop1}
Let $a:(0,\I)\rightarrow \R^+$, $a(r)\in C^\I((0,\I))$ and $a(r)\rightarrow \I$ as $r\rightarrow 0^+$. Then there exists a function $R(r)$, $R\in C^0(\R)$, such that the equation $\dot{x}=R(r)+x^2$ has a saddle-node bifurcation at $r=0$ with no equilibria for $r>0$ and two equilibria for $r<0$. The scaling law for this saddle-node is given by $t\sim a(r)$ as $r\rightarrow 0^+$. 
\end{prop}

\begin{proof}
Define $R(r)=1/a(r)^2$ where we use the even extension for $1/a(r)^2$ to define $R(r)$ on the entire real line. Note that $R(0)=0$ and $R\in C^0(\R)$. Using the same separation of variables argument as in (\ref{eq:sov0})-(\ref{eq:sov3}) it follows that the scaling law is $t\sim a(r)$ as $r\rightarrow 0^+$.
\end{proof}

Proposition \ref{prop:dgprop1} implies that a wide variety of scaling laws can occur if the parameter dependence on $r$ is non-linear and only $C^0$. This is in contrast to the fact that in the smooth case we expect particular power laws. For the rest of this paper we shall focus on the equation:
\be
\dot{x}=r+F(x) \qquad \text{for $x \in\R$ and $r\in\R$} \nonumber
\ee
with $F\in C^\I(\R-\{0\})$ and $F\in C^0(\R)$. We shall show that this case precisely gives an interesting ``intermediate'' behavior between the cases considered so far. We introduce the notation:
\be
\label{eq:nseq}
\dot{x}=\left\{
\begin{array}{r l}
r+f_-(x) & \text{for $x<0$} \\
r+f_+(x) & \text{for $x\geq0$}               
\end{array}
\right.
\ee
with $f_\pm(0)=0$, $f_\pm(x)>0$ for $x\neq0$, $f_-\in C^\I((-\I,0))$ and $f_+\in C^\I((0,\I))$. Observe that (\ref{eq:nseq}) has a saddle-node bifurcation at $(r,x)=(0,0)$. If we use the same techniques to investigate the scaling law for $r\rightarrow 0^+$ as in (\ref{eq:sov0})-(\ref{eq:sov3}) we obtain two integrals:
\be
A_-(r)=\int_{-1}^0 \frac{1}{r+f_-(x)}dx \qquad \text{and} \qquad A_+(r)=\int_0^1 \frac{1}{r+f_+(x)}dx \nonumber
\ee 
Observe that if $f_-(x)\ll f_+(x)$ for $x\rightarrow 0$ then $A_-(r)\ll A_+(r)$ for $r\rightarrow 0^+$. So we can assume that $f_-(x)\sim f_+(x)$ as $x\rightarrow 0$. In particular we assume without loss of generality (with respect to investigating the scaling law) that $f_+(x)$ is given and $f_-(x)$ is its even extension to $(-\I,0)$. For simplicity of notation we shall simply denote $f_+(x)=f(x)$ for $x\geq 0$ and $f_-(x)=f(x)$ for $x<0$, drop the subscripts $\pm$.

\subsection{Three Examples}
We compare three key examples of ODEs $\dot{x}=r+f_i(x)$ $(i=1,2,3)$ leading to further investigation. They are given by
\bea
\label{eq:3ex}
f_1(x)=\sqrt{x} \qquad  f_2(x)=x \qquad f_3(x)=x^2 \qquad \text{for $x\geq 0$}
\eea
where we use even extensions to define the functions $f_i$ on $\R$. The examples are illustrated in Figure \ref{fig:figure1}.

\begin{figure}[htbp]
	\centering
		\includegraphics[width=1.00\textwidth]{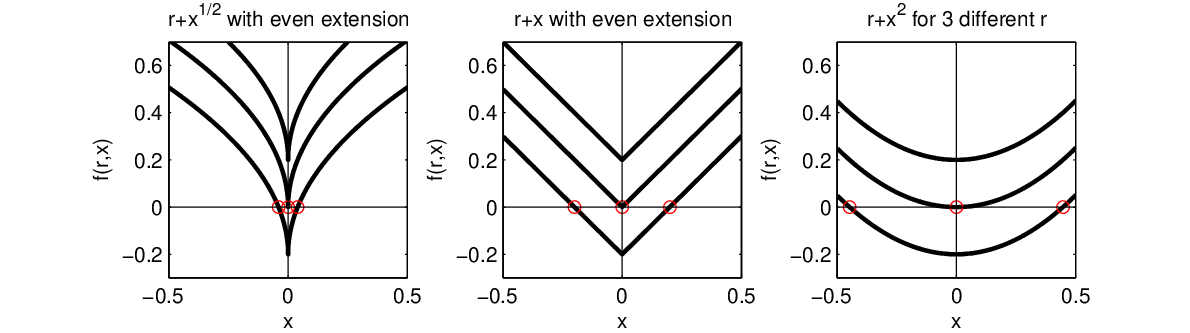}
	\caption{$f(x,r)=r+x^a$ for $r=-\frac15,0,\frac15$ and $a=\frac12,1,2$}
	\label{fig:figure1}
\end{figure}

Integrals giving the scaling laws are:
\be
A_i(r)=\int_0^1\frac{1}{r+f_i(x)}dx \qquad \text{for $i=1,2,3$} \nonumber
\ee
The integrals can be evaluated explicitly and the results are:
\be
A_1(r)= 2+2r\ln(r)-2r\ln(1+r) \qquad A_2(r)=ln\left(1+\frac1r \right) \qquad A_3(r)=\frac{\tan^{-1}\left(\frac{1}{\sqrt{r}}\right)}{\sqrt{r}}\nonumber
\ee
We consider the limit $r\rightarrow 0^+$ to get the scaling laws. The results are summarized in Table \ref{tab:3ex}.\\

\begin{table}[htbp]
	\caption{Scaling laws for three examples}
	\centering
	\label{tab:3ex}
	\begin{tabular}{|l|l|l|}
	\hline
	Function & Scaling Law & Order of the Scaling \\
	\hline\hline
	$f_1(x)=\sqrt{x}$ & $t=2$                      & constant  \\
	$f_2(x)=x$        & $t\sim\ln(r)$              & logarithmic  \\
	$f_3(x)=x^2$      & $t\sim \frac{1}{\sqrt{r}}$ & square-root  \\
	\hline
	\end{tabular}
\end{table}

Therefore we should pose the question, what exponents $\alpha$ for the family $f_\alpha(r,x)=r+x^\alpha$ have scaling laws $O(1)$.

\subsection{Rapid Changes of the Scaling Law}

\begin{thm}
\label{thm:qb}
For each $r$, define $f_\alpha(r,x)$ as the even extension of 
\be
\label{eq:family}
x\mapsto r+x^\alpha \qquad x\geq0 \nonumber
\ee
Consider the two-parameter family of differential equations given by
\be
\dot{x}=f_\alpha(r,x) \qquad \text{for $x\in \R$ and $r\in \R$} \nonumber
\ee
then we find the scaling law for $\alpha\in(0,1)$ to be given by $t=O(1)$, whereas $t\rightarrow\I$ for $\alpha\in [1,\I)$, as $r\rightarrow 0^+$.   
\end{thm}  

\begin{proof}
Let $r_n>0$ be a sequence such that $r_n>r_{n+1}$ for all $n$ and $r_n\rightarrow 0$ as $n\rightarrow \I$. Then consider 
\be
\label{eq:mct}
\lim_{n\rightarrow \I}\int_0^1 \frac{1}{r_n+x^\alpha}dx
\ee
Notice that $g_n(x)=\frac{1}{r_n+x^\alpha}$ is a monotonically increasing sequence of functions on $[0,1]$. Therefore we can apply the monotone convergence theorem in equation (\ref{eq:mct}) to obtain:
\be
\lim_{n\rightarrow \I}\int_0^1 \frac{1}{r_n+x^\alpha}dx=\int_0^1 \lim_{n\rightarrow \I} \frac{1}{r_n+x^\alpha}dx= \int_0^1 \frac{1}{x^\alpha}dx \nonumber
\ee
where the last integral converges for $\alpha\in(0,1)$ and diverges for $\alpha\geq 1$.
\end{proof}

The main observation of Theorem \ref{thm:qb} is that there is a major quantitative change in the behavior of the solutions at $\alpha=1$. In analogy with the classical terminology we refer to $\alpha=1$ as a ``quantitative bifurcation''. Notice that $\alpha=1$ lies on the boundary of $C^0$ and $C^1$ functions spaces for the family $f_\alpha(r,x)$.\\

It is an easy extension of Proposition \ref{prop:dgprop2} that for $\alpha>1$ the scaling laws follow powers of $r$, i.e. $t\sim r^{j}$ with $j\in (0,1)$. We have seen in Table \ref{tab:3ex} that for $\alpha=1$ the scaling is logarithmic. Hence the situation observed near $\alpha=1$, providing a transition from a power law to a constant via a logarithmic ``break-point'', resembles the situation of classical qualitative bifurcation theory in the context of a quantitative feature of a dynamical system.    

\section{A Model Problem}

Consider a pendulum as shown in Figure \ref{fig:pend}.\\

\begin{figure}[htbp]
\psfrag{zero}{$0$}
\psfrag{-pi2}{$-\frac{\pi}{2}$}
\psfrag{pi2}{$\frac\pi2$}
\psfrag{pi}{$\pi$}
\psfrag{L}{$l$}
\psfrag{g}{$g$}
\psfrag{t}{$\theta$}
\psfrag{m}{$m$}
	\centering
		\includegraphics[width=0.5\textwidth]{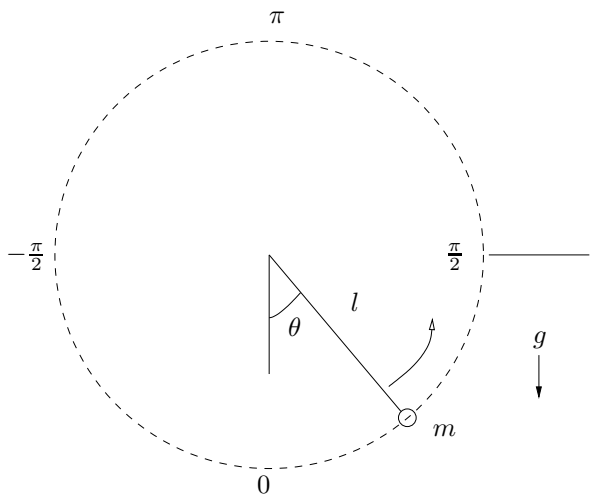}
	\caption{Illustration of the simple pendulum with a barrier at $\frac\pi2$ at distance 1 from the center.}
	\label{fig:pend}
\end{figure}

 We use the following notation: $l$ is the length of the pendulum, $m$ is its mass, $g$ is acceleration due to gravity, $\nu$ is a viscous damping coefficient, $\omega$ denotes constant forcing, $\theta$ is the angle and $I$ denotes the moment of inertia of the pendulum. Then Newton's law gives the equation of motion as:
\be
I\theta''+mgl \sin \theta =\nu (\omega-\theta')  \nonumber
\ee
If we consider the overdamped limit of very large damping we can neglect the term $I\theta''$ and consider the classical nonuniform oscillator
\be
\label{eq:main1}
\theta'=\omega-\frac{mgl}{\nu}\sin\theta
\ee
For simplicity let us set $L=mgl/\nu$, then (\ref{eq:main1}) reads:
\be
\theta'=\omega-L\sin\theta \nonumber
\ee 
Usually one assumes that the length of the pendulum is fixed and the bifurcation parameter is the applied torque. We generalize this approach and assume that we can vary the length $L$ depending on the angle of the pendulum, i.e. $L=L(\theta)$ with $L:S^1 \rightarrow \R^+$. Note that we regard $S^1$ as $[-\pi,\pi]$ with endpoints identified as indicated in Figure \ref{fig:pend}. If we assume that we want to vary the length symmetrically on $S^1$ there will be breakpoints for $L(\theta)$ at $\theta=\pi,\frac{\pi}{2},0,-\frac{\pi}{2}$. Furthermore let us assume that the length of the pendulum is limited at one of the breakpoints, say $L(\pi/2)=1$ (see Figure \ref{fig:pend}). Let us consider the family of functions:
\be
F_a(\theta)=\left\{\begin{array}{r l}
-1+(-2/\pi(\theta+\pi/2))^a  & \text{for $\theta\in[-\pi,-\frac\pi2)$} \\
-1+(2/\pi(\theta+\pi/2))^a  & \text{for $\theta\in[-\frac\pi2,0)$} \\
1-(-2/\pi(\theta-\pi/2))^a  & \text{for $\theta\in[0,\frac\pi2)$} \\
1-(2/\pi(\theta-\pi/2))^a  & \text{for $\theta\in[\frac\pi2,\pi]$}               
\end{array} \nonumber
\right.
\ee
for $a>0$. Illustrations for $a=\frac12,1,2$ are given in Figure \ref{fig:figure2}.\\
\begin{figure}[htbp]
	\centering
		\includegraphics[width=1.00\textwidth]{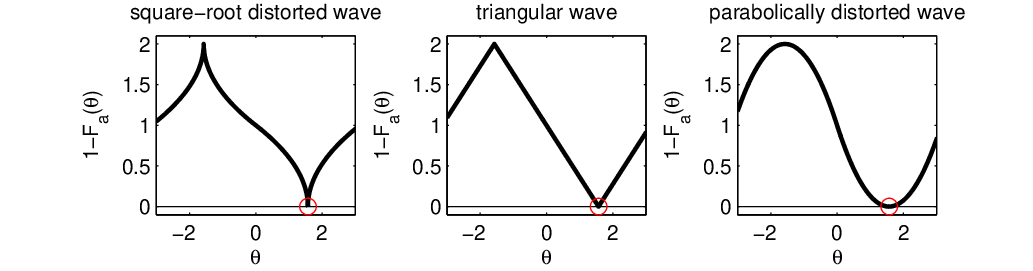}
	\caption{$f(\theta,1)=1+F_a(\theta)$ for $a=\frac12,1,2$}
	\label{fig:figure2}
\end{figure}

Now we consider the equation modeling the elongation $L$ as: 
\be
L(\theta)=\frac{F_a(\theta)}{\sin\theta} \nonumber
\ee
Note that this is a priori not defined for $a<1$ at $0$ giving infinite length $L$. We can simply truncate $L(\theta)$ if we want to construct the experiment but we remark that the form of $L(\theta)$ near $0$ is not relevant for the following discussion. It is crucial to notice that $L(\pi/2)=1$ independent of $a$, i.e. we pass the bottleneck with the same length. We are left with the equation 
\be
\label{eq:dnuo}
\dot{\theta}=\omega-F_a(\theta) \qquad \text{for $\theta\in[-\pi,\pi]$ and $a,\omega\in\R^+$}
\ee
As in the usual sinusoidal case $\dot{\theta}=\omega-\sin\theta$ we have a saddle-node bifurcation at $(\omega,\theta)=(1,\frac\pi2)$ for all $a>0$ in equation (\ref{eq:dnuo}) (see also Figure \ref{fig:figure2}). From Theorem \ref{thm:qb} we see that there exists a quantitative bifurcation for $a=1$. In particular, $a\in(0,1)$ gives scaling laws $t\sim O(1)$ as $\omega\rightarrow 1^+$. For the triangular wave $a=1$, we get a logarithmic scaling and for $a>1$ we obtain power laws $t\sim (\omega-1)^j$ for $j\in (0,1)$.\\

This means that upon tuning the parameter $a$ in the family $F_a$ and setting $\omega$ small and positive we can switch the behavior of the bottleneck. Hence equation (\ref{eq:dnuo}) describes how different strategies of varying $L$ affect the motion. The case $0<a<1$ corresponds to reducing a long pendulum to unit length for $\theta\rightarrow(\frac{\pi}{2})^-$, the case $a=1$ corresponds to a very small variation upon approaching $\frac{\pi}{2}$ and $a>1$ means that we start with a very short pendulum with a rapid increase near $\pi/2$ to reach unit length.

\section{Conclusions}

We have investigated scaling laws for saddle-node bifurcation in 1-dimensional dynamics and have carried out a \textit{systematic study} of different scaling laws if the different conditions of non-degeneracy and smoothness are modified. Dropping the assumptions of non-degeneracy we have seen that the nonlinearities of the parameter or the phase-space variable exhibit various power laws. Furthermore we have investigated the case of $C^0$ vector fields and demonstrated that parameter linearities lead to basically arbitrary scaling laws. For phase space nonlinearities we found a rapid change of the scaling law for a 2-parameter family of $C^0$ vector fields. In particular, the theory presented for $C^0$ vector fields can clearly be extended to more than 1 dimension, but in contrast to the smooth cases we do not have tools like center manifolds immediately available. As an example it is easy to construct an n-dimensional system, which has locally at least two directions near a saddle-node with different $C^0$ vector fields in each direction. This substantially complicates the analysis for the natural extension to more dimensions.\\

We have also demonstrated that a very simple pendulum equation can exhibit a rapid change in the scaling law. Many other applications of saddle-node bifurcations occurring in n-dimensional $C^0$ systems could clearly exhibit this phenomenon. Furthermore the different scaling laws found for $C^0$ vector fields can be directly related to the same scaling laws found in $C^0$ maps. $C^0$ maps are of particular relevance in the analysis of discontinuity-induced bifurcations and their associated return maps (Poincar\'{e} discontinuity map [PDM] and zero time discontinuity map [ZDM]) \cite{PSDS}. These return maps have different types of singularities, among them we can find piecewise-linear maps and maps with square-root singularities. Hence the methods we presented in this paper are likely to be very useful in the analysis of scaling laws for discontinuity-induced bifurcations.\\ 

\textit{Remark}: After finishing the calculations for this paper it was brought to the attention of the author that the detailed calculation of the integral required to prove Proposition \ref{prop:dgprop2} was carried out by Fontich and Sardanyes in \cite{Fontich}. Therefore we feel very confident in omitting this calculation. The example of the autocatalytic replicator model \cite{Fontich} serves as an excellent example for Proposition \ref{prop:dgprop2} which we have not supplemented with an example in our paper.
 
\bibliographystyle{plain}
\bibliography{../my_refs}

\end{document}